\documentclass[12pt]{article}
\usepackage{amsmath}
\usepackage{amssymb}
\usepackage{amsthm}
\usepackage{url}

\numberwithin{equation}{section}
\newtheorem{theorem}{Theorem}[section]
\newtheorem{proposition}[theorem]{Proposition}
\newtheorem{lemma}[theorem]{Lemma}
\newtheorem{Remark}[theorem]{Remark}

\newcommand\dstyle\displaystyle

\newcommand\ma{\medskipamount}

\newcommand\mLP{\\[\ma]}

\newcommand\al\alpha
\newcommand\be\beta
\newcommand\de\delta
\newcommand\la\lambda
\newcommand\tha\theta

\newcommand\iy\infty

\newcommand\bma{\begin{pmatrix}}
\newcommand\ema{\end{pmatrix}}

\begin{document}

\title{Supersymmetric Quantum mechanics \\ on the radial lines  }
\author{  F.  Bouzeffour and M. Garayev \\ Department of mathematics, College of Sciences.\\
 King Saud University, P. O Box $2455$ Riyadh $11451$, Saud Arabia. \\
fbouzaffour@ksu.edu.sa, mgarayev@ksu.edu.sa}
\maketitle
\begin{abstract}We investigate a type of Hermite orthogonal polynomials on $r$ lines in the plane which have a common point at the origin and endpoints at the $r$ roots of unity and we show that their  related Hermite functions are
 eigenfunctions of a differential-difference operator. A supersymmetric harmonic
oscillator on $r$ radial lines is presented and analyzed. Its eigenfunctions are given in terms of these polynomials.
\end{abstract}
\section{Introduction}
In \cite{Zhedanov,P1,P2}, the authors  formulated a supersymmetry quantum mechanics for one-dimensional systems by using difference-differential operators  known in the literature as Dunkl operators \cite{19,18}.  One of its characteristic features is that both a supersymmetric Hamiltonian and a supercharge component involve reflection operators. In addition their related wave functions are expressed in terms of Hermite orthogonal polynomials $H_n(x)$, which are orthogonal polynomials over the real line $\mathbb{R}$ with respect to the weight function $w(x) = e^{-x^2}$, so that
\begin{equation}
\int_{-\infty}^\infty H_n(x)x^ke^{-x^2}\,dx=0,\quad k=0,\,\dots,\,n-1.
\end{equation}  An ordinary supersymmetric quantum-mechanical system may be generated by three operators $Q,$ $Q^\dag$ and $H$ satisfying \cite{1,P2}
\begin{equation}
H=QQ^\dag+Q^\dag Q \quad \mbox{and}\quad Q^2=Q^{\dagger^2}=0.
\end{equation}
The basic ingredient of the supersymmtric quantum mechanics  \cite{} is the $Z_2$ grading operator $\mathrm{\Gamma},$ $\mathrm{\Gamma}^2=1$, which classifies all the operators into even (bosonic, $b$)
and odd (fermonic,$f$) subsets accordingly to the relations $$[\mathrm{\Gamma},b]
=\{\mathrm{\Gamma},f\}=0.$$A realization of this algebra is formulated by taking the following supercharge
$$Q=\frac{1}{\sqrt{2}}(-id/dx-iW)f$$
where $W=W(x)$ is a superpotential and $f,$ $f^\dagger$ are fermionic annihilation and creation operators satisfying
\begin{equation*} f^2=(f^\dagger)^2=0, \qquad \{ f, f^\dagger\}=1, \end{equation*}
and represented by the $2\times 2$ matrices:
\begin{equation*}
   f=
  \left[ {\begin{array}{cc}
   0 & 1 \\
   0 & 0 \\
  \end{array} } \right]
  \qquad f^\dagger=
  \left[ {\begin{array}{cc}
   0 & 0 \\
   1 & 0 \\
  \end{array} } \right]
.\end{equation*}
The supersymmetric Hamiltonian is given by
\begin{equation*}\label{25} H=\{ Q, Q^\dagger\}=\frac{1}{2}(\frac{d^2}{dx^2}+W^2)+\frac{1}{2}\frac{dW}{dx}\sigma_3\end{equation*}
where
\begin{equation*} \sigma_3=[f, f^\dagger]=
  \left[ {\begin{array}{cc}
   1 & 0 \\
   0 & -1 \\
  \end{array} } \right].\end{equation*}
Another realization of supersymmetric quantum mechanics, called minimally bosonized supersymmetric quantum (  it does not
involve the presence of "spin-like" degrees of freedom)  is to take the reflection operator $ (R\psi)(x)=\psi(-x)$, as a grading operator. With the supercharges $Q$ and $Q^\dagger$ having the expression
\begin{equation}
Q=\frac{1}{\sqrt{2}}(\frac{d}{dx}+U(x))R+\frac{1}{\sqrt{2}}V(x)
\end{equation}
where $U(x)$ is an even function and $V(x)$ is an odd function.
The following Hamiltonian $H$, given by
\begin{align}
&H=-\frac{1}{2}\frac{d^2}{dx^2}+\frac{1}{2}(U^2+V^2)+\frac{1}{2}\frac{dU}{dx}-\frac{dV}{dx}R\label{hebel}
\end{align}
In \cite{Zhedanov}, the authors considered  a simple example of supersymmetric quantum mechanics given by the following Hamiltonian:
 \begin{equation}
H=Q^2=-\frac{1}{2}\frac{d^2}{dx^2}+\frac{1}{2}\,x^2-\frac{1}{2}R
\end{equation}
where $R$ is the reflection operator acting as: $ (R\psi)(x)=\psi(-x),$ and  the supercharge $Q$ is given by
\begin{equation}
Q=\frac{1}{\sqrt{2}}(\frac{d}{dx}R+x)\label{char1}
\end{equation}
 \begin{figure}[ht]
\unitlength=0.6mm
\centering
\begin{picture}(100,100)(0,0)
\put(50,50){\circle*{3}}
\put(100,50){\circle*{3}}
\put(65.45,97.55){\circle*{3}}
\put(9.55,79.39){\circle*{3}}
\put(9.55,20.61){\circle*{3}}
\put(65.45,2.45){\circle*{3}}
\put(54,53){$0$}
\put(104,53){$1$}
\put(70,100){$\omega_r$}
\put(5,84){$\omega_r^2$}
\put(5,12){$\omega_r^3$}
\put(70,0){$\omega_r^4$}
\put(50,50){\qbezier(-50,0)(25,0)(60,0)}
\put(50,50){\qbezier(-15,-47,55)(7.5,23.775)(15,47.55)}
\put(50,50){\qbezier(-15,47.55)(7.6,-23.775)(15,-47.55)}
\put(50,50){\qbezier(40.45,-29.39)(-20.225,14.695)(-40.45,29.39)}
\put(50,50){\qbezier(40.45,29.39)(-20.225,-14.695)(-40.45,-29.39)}
\end{picture}
\caption{ $r=5$ with $\omega_r=e^{2\pi i/5}$.}
\label{fig:rstar}
\end{figure}
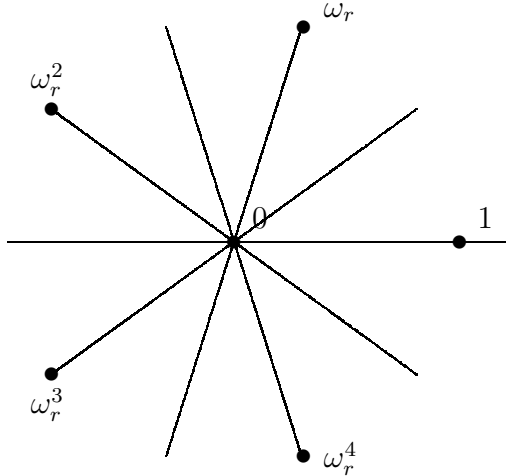
In this paper we will extend the Hermite polynomials to polynomials on $r$ lines. We consider a special configuration for the $r$ lines by having one common point $0$ and passing through the $r$ roots of unity $\omega_r^{j}$,  $j=0,\,\dots,\,r-1$ with $\omega_r=e^{\frac{2i\pi}{r}}$. see Figure 1.\\
To preserve the symmetry, we take a weight function $w(x)=|x|^{2\nu}e^{-x^{2r}}$ and a
measure $\mu_j$, which is supported on the line $\delta_j=\omega^j_r \mathbb{R}$, $j = 0,\,\dots,\, r-1$ with $ w(x)$ as its Radon-Nikodym derivative.  The orthogonality relations for the type Hermite polynomials $H_N^{(r,\nu)}(x)$ on the radial lines
are then given by \begin{equation}
\sum_{j=0}^{r-1}\omega_r^{-j}\int_{\omega^{j}_r\mathbb{R}}x^k\overline{H^{(r,\nu)}_N(x)} |x|^{2\nu}e^{-x^{2r}}dx=0,\quad k=0,\,\dots, N-1.\label{pr1}
\end{equation}
In particular when $r=1$ and $\nu=0$, we have the standard case  of the  Hermite polynomials $H_n$ defined in \eqref{hermout1}. It is interesting to know what kind of supersymmetric Hamiltonian involving reflection operators admits exact eigenfunctions which are expressible in terms of the Hermite type  orthogonal polynomial on the radial lines.
In the literature we found various type of polynomials which are orthogonal polynomials on  radial lines in the complex plane, see for instance \cite{M1}. In particular, Milovanovi\'{c}  studied  a generalized Hermite polynomials related to the following inner product \cite{M}  \begin{equation*}
\left\langle f,g\right\rangle=
\sum_{j=0}^{2r-1}\int_{0}^\infty f(\varepsilon_r^{j}x)\overline{g(\varepsilon_r^{j}x)} |x|^{2r\alpha}e^{-x^{2r}}dx,\, \varepsilon_r=e^{\frac{i\pi}{r}}.
\end{equation*}
We study the particular when $r=1$ in section 2 and in section 3  we investigate the generalized Hermite polynomials on the radial lines. In section 4 we
give a new Dunkl type operator, which intrinsically connect with the earlier sections.
Using this operator a supersymmetric harmonic oscillator on $r$ radial lines is presented and analyzed. Its eigenfunctions are given in terms of the Hermite polynomials on the radial lines.
\section{Generalized Hermite polynomials on the real line}
Recall that the $R$-deformed Heisenberg algebra is generated by  $a$, $a^\dagger$ and $R$ 
\begin{equation}
[a,a^\dagger]=1+2\nu R,\quad \{R,a\}=\{R,a^\dagger\}=0,\quad R^2=1.
\end{equation}
which possesses unitary infinite-dimensional representations for $\nu > -1$
acts as follows on the sates $|n,\nu\rangle$
\begin{align}
\begin{cases}
&a|n,\nu\rangle=\sqrt{n+\,\theta_n}|n-1,\nu\rangle,\\
&a^\dagger|n,\nu\rangle=\sqrt{n+1+\,\theta_{n+1}}|n+1,\nu\rangle,\\
&R|n,\nu\rangle=(-1)^n|n,\nu\rangle,
\end{cases}
\end{align}
where
\begin{align}\label{tt}
 \theta_n = \begin{cases}0 & \mbox{ if } n \mbox{ is even}\\ 2\nu&\mbox{ if } n \mbox{ is odd},\end{cases}
\end{align}
A realization of this algebra is given by the operators
\begin{equation}\label{H4}
a=\frac{1}{\sqrt{2}}(x+D_\nu),\quad a^\dagger=\frac{1}{\sqrt{2}}(x-D_\nu).
\end{equation}
where $D_\nu$ is the  Yang-Dunkl operator
 \begin{equation}
D_\nu=\frac{d}{dx}+\frac{\nu}{x}(1-R).
\end{equation}
The bosonic oscillator Hamiltonian associated to the $R$-deformed Heisenberg algebra is given by 
\begin{equation}
H_0=-\frac{1}{2}\frac{d^2}{dx^2}-\frac{\nu}x\,\frac{d}{dx}\,+\frac{\nu}{2x^2}
(1-R)+\frac{1}{2}x^2.
\end{equation}
The wave function corresponding to the well-known eigenvalue
\begin{equation}
\lambda_n=n+\nu+\frac{1}{2},\quad n=0,\,1,\,2,\dots
\end{equation}
are given by
\begin{equation}
\psi_n(x)=(-1)^{[n/2]}2^n([n/2]!\Gamma([(n+1)/2]+\nu+1/2) )^{-1/2}e^{-x^2/2}H^{(\nu)}_n(x),
\end{equation}
where $H^{(\nu)}_n(x)$ is the generalized Hermite polynomials. It is well known that for $\nu>-\frac{1}{2},$ these polynomial can be expressed in terms of the
Laguerre polynomial  $L_n^\nu(x)$
 \begin{equation} \begin{cases}
     H_{2n}^{(\nu)}(x) = (-1)^n 2^{2n} n! \,L_n^{\nu-\frac{1}{2}}(x^2),\\
     H_{2n+1}^{(\nu)}(x) =(-1)^n 2^{2n+1} n! \,xL_n^{\nu+\frac{1}{2}}(x^2).
   \end{cases}.\label{hermout1}\end{equation}
They satisfy the orthogonality relations
:\begin{equation}
\int_{\mathbb{R}}H^{(\nu)}_n(x)
H^{(\nu)}_m(x)|x|^{2\nu}e^{-x^2}\,dx=\gamma_n^{-1}\delta_{n\,m},\label{orth}
\end{equation}
where \begin{equation}\gamma_n^{-1}=2^{2n}
\Gamma([\frac{n}{2}]+1)\Gamma([\frac{n+1}{2}]+
\nu+\frac{1}{2}),\,\,n=0,\,\dots\label{constant3} \end{equation} and $[x]$
denotes the greatest integer function.\\We introduce the following supercharge operator $Q$
\begin{equation}
Q=\frac{1}{\sqrt{2}}\big(D_\nu R+x\big)
\end{equation}
After evaluating $Q^2$, we get the following form for a supersymmetric Hamiltonian
\begin{equation}
H=Q^2=-\frac{1}{2}\frac{d^2}{dx^2}-\frac{\nu}x\,\frac{d}{dx}\,+\frac{\nu}{2x^2}
(1-R)+\frac{1}{2}x^2-\frac{1}{2}R-\nu.
\end{equation}
The spectrum of $H$ is easily obtained by observing that
\begin{equation}
H=H_0-\frac{1}{2}R-\nu.
\end{equation}
It follows
\begin{equation}
H|n,\nu\rangle=(n+\frac{1}{2}(1-(-1)^n)|n,\nu\rangle.
\end{equation}
The spectrum of $H$  consist only for the even number  starting with zero. Each level is degenerate except for the ground states which is unique.

\section{Hermite orthogonal polynomials on the radial lines}
Let $r$ be a fixed odd integer $r\geq 2,$ and $\omega_r$ be a $r$th  primitive root of unity, i.e, $\omega_r=e^{\frac{2i\pi}{r}}.$ 
In this section we investigate orthogonal polynomials relative to the following inner product
\begin{align}  \label{e10}
(f,g)_\nu&=\sum_{j=0}^{r-1}\omega_r^{-j}\int_{\delta_j }f(x)\overline{g(x)} |x|^{2\nu}e^{-x^{2r}}dx,\\&=\sum_{j=0}^{r-1}\int_{\mathbb{R} }f(\omega_r^{j}x)\overline{g(\omega_r^{j}x)} |x|^{2\nu}e^{-x^{2r}}dx.
\end{align}
Observe that
\begin{equation*}
\left\langle f,f\right\rangle _{\nu}=
\sum_{j=0}^{r-1}\int_{\mathbb{R}}|f(\omega_r^jx)|^2
|x|^{2\nu}e^{-x^{2r}}\,dx>0,
\end{equation*}
except when $f=0,$ then \eqref{e10} define an inner product.
Following the steps of \cite{M,M1}, we can prove the existence of orthogonal polynomials $H^{(r,\nu)}_N(x)$ such that
\begin{equation}
\sum_{j=0}^{r-1}\int_{\mathbb{R} }(\omega_r^jx)^k\overline{H^{(r,\nu)}_N(\omega_r^jx)} |x|^{2\nu}e^{-x^{2r}}dx=0,\quad k=0,\,\dots, N-1.\label{pr1}
\end{equation}
We normalize these polynomials so that the coefficient of $x^N$ in $H^{(r,\nu)}_N(x)$
 is $2^{[\frac{N}{r}]}.$
\begin{proposition}We have
\begin{equation}
H^{(r,\nu)}_N(\omega_r x)=\omega_r^{N}H^{(r,\nu)}_N(x)
\end{equation}
\end{proposition}

\begin{proof}Let $q(x)$ is an arbitrary polynomial of degree at most $N-1,$ we have
\begin{align*}
(q(x),H_N^{(r,\nu)}(\omega_rx))_\nu&=\sum_{j=0}^{r-1}\int_{\mathbb{R} }q(\omega_r^jx)\overline{H^{(r,\nu)}_N(\omega_r^{1+j}x)} |x|^{2\nu}e^{-x^{2r}}dx,\\&
=\sum_{j=0}^{r-1}\int_{\mathbb{R} }q(\varepsilon^j_rx)\overline{H^{(r,\nu)}_N(\omega_r^{1+j}x)} |x|^{\alpha}e^{-x^{2r}}dx\\&
=\sum_{j=1}^{r}\int_{\mathbb{R} }q(\varepsilon^{2j-2}_rx)\overline{H^{(r,\nu)}_N(\omega_r^{j}x)} |x|^{2\nu}e^{-x^{2r}}dx\\&
=\sum_{j=0}^{r-1}\int_{\mathbb{R} }q(\omega_r^{1-j}x)\overline{H^{(r,\nu)}_N(\omega_r^jx)} |x|^{2\nu}e^{-x^{2r}}dx\\&
=0.
\end{align*}
Since the degree of $q(x)$ is less than $N-1$ , then  $\omega_r^{-N}H^{(r,\nu)}_N(\omega_rx)$
is an orthogonal polynomial with respect to \eqref{pr1} . Finally, from the uniqueness of $H^{(r,\nu)}_N(x)$ it follows that
\begin{equation}
H^{(r,\nu)}_N(\omega_r x)=\omega_r^{N}H^{(r,\nu)}_N(x).
\end{equation}
\end{proof}

\begin{theorem}The  Hermite type polynomials $\{H^{(r,\nu)}_N(x)\}$ satisfy the following orthogonality relations
\begin{equation}  \label{e9}
\sum_{j=0}^{r-1}\int_{\mathbb{R}}H^{(r,\nu)}_N(x)
\overline{H^{(r,\nu)}_M(x)} |x|^{2\nu}e^{-x^{2r}}dx=\zeta_N\delta_{NM},
\end{equation}
where \begin{equation}\zeta_N=2^{[N/r]} \Gamma([\frac{N}{2r}]+1)\Gamma([\frac{N+r}{2r}]+\frac{2\nu+2+2s-r}{2r}).\label{z1}
\end{equation}Furthermore, the polynomials $\{H^{(r,\nu)}_N(x)\}$ can be expressed in terms of the generalized Hermite polynomial
\begin{equation}
H_N^{(r,\nu)}(x)=x^sH_n^{(\nu_s)}(x^r),\quad N=nr+s
\end{equation}
where $ s=0,\dots,\,r-1\,\,$ and
\begin{equation}
\nu_s=\frac{2\nu+2s+1-r}{2r}.
\end{equation}
\end{theorem}
\begin{proof}From Proposition 3.1, the polynomial $H_N^{(r,\nu)}(x)$ can be written uniquely in the form
\begin{equation}
H_N^{(r,\nu)}(x)=x^sQ_n(x^r),\quad N=nr+s.
\end{equation}
Let $M=mr+t$, $t=0,\,\dots,r-1$, we have
\begin{align}  \label{e8}
(H_N^{(r,\nu)},H_M^{(r,\nu)})_\nu&=\sum_{j=0}^{r-1}\omega_r^{j(s-t)}\int_{\mathbb{R} }Q_n(x^r)\overline{Q_m(x^r)}x^sx^t |x|^{2\nu}e^{-x^{2r}}dx,\\&
=\delta_{st}r\int_{\mathbb{R} }Q_n(x^r)\overline{Q_m(x^r)} |x|^{2\nu+2s}e^{-x^{2r}}dx\\&=\delta_{st}\int_{\mathbb{R} }Q_n(x)\overline{Q_m(x)} |x|^{\frac{2\nu+2s+1-r}{r}}e^{-x^{2}}dx.
\end{align}
The existence of the polynomials $\{H_N^{(r,\nu)}(x)\}$
is reduced to the existence of polynomials orthogonal on $\mathbb{R}$ with
respect to the weight function $|x|^{2\nu_s}e^{-x^{2r}}.$\\ From \eqref{orth} we obtain
\begin{equation}
H_N^{(r,\nu)}=x^sH^{(\nu_s)}_n(x^r),
\end{equation}
where
\begin{equation}
\nu_s=\frac{2\nu+2s+1-r}{2r}, \quad s=0,\dots,\,r-1.
\end{equation}
\end{proof}
\begin{theorem}The Hermite polynomials $\{H_N^{(r,\nu)}(x)\}$ on the radial lines satisfy the three terms recurrence relations
\begin{align}
&2x^r\,H_N^{(r,\nu)}(x)=H_{N+r}^{(r,\nu)}(x)-2\big([N/r]+
\vartheta_N\big)H_{N-r}^{(r,\nu)}(x),N\geq r \label {rec1}\\&
H_N^{(r,\nu)}(x)=x^N,\quad N=0,\,\dots,\,r-1.
\end{align}

\end{theorem}
\begin{proof}The recurrence relation \eqref{rec1} follows from the following recurrence relation for the generalized Hermite polynomials $H^{(\nu)}_n(x)$ \cite{Chihara}
\begin{align}
&H^{(\nu)}_{n+1}(x)=2xH^{(\nu)}_n(x)-2(n+\theta_n)H^{(\nu)}_{n-1}(x)\\&
H^{(\nu)}_{-1}(x)=0,\quad H^{(\nu)}_0(x)=1.
\end{align}
where $\theta_n$ is defined in \eqref{tt}

\end{proof}
\section{Dunkl harmonic oscillator on the radial lines}
Let $\mathcal{P}$ be the real vector space of all polynomials in one variable with real coefficients. For each odd integer $r$, we denote by $s_r$ the complex reflection acting on  $f\in \mathcal{P}$  as
$$
(s_rf)(x):=f(\varepsilon_r x), \,\,\varepsilon_r=e^{\frac{i\pi}{r}},
$$
and by  $\Pi_0(r),\,\Pi_1(r),\dots,\,\Pi_{r-1}(r)$  the orthogonal projections related to the complex reflection $s_r$, which are defined by
\begin{equation}
\Pi_{i}(r)=\frac{1}{r}\sum_{j=0}^{r-1}\omega_r^{-ij}s_r^j,\quad i=0,\,\dots,\,r-1.
\end{equation}
They constitute a system of resolution of the identity
\begin{equation}
\Pi_0(r)+\Pi_1(r)+\dots+\Pi_{r-1}(r)=1,\quad \mbox{and}\quad \Pi_i(r)\Pi_j(r)=\delta_{ij}\Pi_i(r) .
\end{equation}
Let $\nu >\frac{r-1}{2}$ and consider the following differential-difference operator
\begin{equation}
Y_\nu=\frac{d}{dx^r}+\frac{1}{rx^r}\sum_{s=0}^{r-1}(2\nu+1+s-r)\Pi_{r+s}(2r)
-s\Pi_s(2r),
\end{equation}
where
\begin{equation}
\frac{d}{dx^r}=\frac{1}{rx^{r-1}}\frac{d}{dx}.
\end{equation}
The operator $Y_\nu$ acts on monomials $x^k$ as
\begin{align*}&Y_\nu x^{s}=0,\\&Y_\nu x^{2nr+s}=2nx^{2(n-1)+r+s}, \\&Y_\nu x^{2nr+r+s}=(2n+2\nu_s+1)x^{2nr+s},
\end{align*}
where $s=0,\,\dots,\,r-1,$ and
\begin{equation}\label{num1}
\nu_s=\frac{2\nu+2s+1-r}{2r}.
\end{equation}
We introduce  the deformed number $[N]_\nu$ for each  $N=0,\,1,\,\dots$ by
\begin{equation}
[N]_\nu=[N/r]+\vartheta_N,\label{num}
\end{equation}
where \begin{equation}
\vartheta_N=\begin{cases}
&\displaystyle 0,\quad \mbox{if}\quad N=2nr+s,\, s=0,\dots, r-1\mLP
&\displaystyle
2\nu_s \quad \mbox{if}\quad N=2nr+r+s, \,\, s=0,\dots, r-1.
\end{cases}
\label{2}
\end{equation}
Obviously
\begin{equation}
Y_\nu x^N=[N]_\nu  x^{N-r}.
\end{equation}
Notice the important property of the operator $Y_\nu$ that is: it sends the linear space of polynomials of degree less than $n $ to the space of dimension $n-r$. In particular, this means that there are no polynomial eigenfunction of this operator.
\begin{proposition}We have
\begin{equation}
Y_\nu H^{(\nu)}_N(x)=2[N]_\nu H^{(\nu)}_{N-r}(x)\label{rec2}.
\end{equation}
\end{proposition}
\begin{proof}We need the following relations for Laguerre polynomials \cite{18}
\begin{align}
& x\frac{dL_n^{\alpha}(x)}{dx}=nL_n^{\alpha}(x)-(n+\alpha)L_{n-1}^{\alpha}(x)\label{La1}\\&
L_{n}^{\alpha-1}(x)=L_{n}^{\alpha}(x)-L_{n-1}^{\alpha}(x)\label{La2}\\&
\frac{dL_n^{(\al)}(x)}{dx}=-L_{n-1}^{(\al+1)}(x)\label{La3}.
\end{align}
\\If $N=2nr+s$ with $0\leq s\leq r-1$, then by \eqref{La3} we have
\begin{align*}
\frac{dH^{(r,\nu)}_N(x)}{dx}&=(-1)^n2^{2n} n!
sx^{s-1}L_n^{(\nu_s-1/2)}(x^{2r})+
(-1)^{n-1}2^{2n+1} n!rx^{2r+s-1}L_{n-1}^{(\nu_s+1/2)}(x^{2r})\\&=\frac{s}{x}H^{(r,\nu)}_N(x)+4nrx^{r-1}
H^{(r,\nu)}_{N-r}(x).
\end{align*}
Hence
\begin{align}
\frac{dH^{(r,\nu)}_N(x)}{dx^r}-\frac{s}{rx^r}H^{(r,\nu)}_N(x)=4nH^{(r,\nu)}_{N-r}(x).
\end{align}
 Combine  \eqref{La1} and \eqref{La2} to get
\begin{align}
& x\frac{dL_n^{\alpha}(x)}{dx}=(n+\alpha)L_n^{\alpha-1}(x)-\alpha L_{n}^{\alpha}(x).\label{La4}
\end{align} If $N=2nr+r+s$ with $0\leq s\leq r-1$, then from \eqref{hermout1} and \eqref{La4} we can write
\begin{align*}
\frac{dH^{(r,\nu)}_{N}(x)}{dx}&=(-1)^n2^{2n+1} n!(x^{r+s}L_n^{(\nu_s+1/2)}(x^{2r}))'\\&=(-1)^n2^{2n+1} n!(r+s)x^{r+s-1}L_n^{(\nu_s+1/2)}(x^{2r})\\&+(-1)^n2^{2n+2} n!
rx^{3r+s-1}\frac{dL_n^{(\nu_s+1/2)}(x^{2r})}{dx}\\&=\frac{r+s-2r(\nu_s+1/2)}{x}H^{(\nu)}_N
(x)+2r(2
n+2\nu_s+1)
x^{r-1}H^{(\nu)}_{N-r}(x).
\end{align*}
Hence
\begin{align}
\frac{dH^{(\nu)}_{N}(x)}{dx^r}+\frac{2\nu+1+s-r}{rx^r}H^{(\nu)}_N(x)=2(2n+2\nu_s+1)
H^{(\nu)}_{N-r}(x).
\end{align}
The result follows from \eqref{num} and the following form for the operator $Y_\nu$
\begin{align*}
Y_\nu=&\sum_{s=0}^{r-1}\big(\frac{d}{dx^r}+\frac{1}{rx^r}(2\nu+1+s-r)\big)\Pi_{r+s}(2r)\\&+
\sum_{s=0}^{r-1}\big(\frac{d}{dx^r}-\frac{s}{rx^r}\big)
\Pi_s(2r).
\end{align*}
\end{proof}

\begin{proposition}The generalized Hermite polynomials satisfy the following differential-difference equation:
\begin{align}
Y^2_\nu H_N^{(r,\nu)}(x)+2 x^rY_\nu \,H_N^{(r,\nu)}(x)=2[N]_\nu H_{N}^{(r,\nu)}(x).
\end{align}
\end{proposition}
\begin{proof}
From the recurrence relation \eqref{rec1} and \eqref{rec2}, we have
\begin{align*}
2x^r\,H_N^{(r,\nu)}(x)&=H_{N+r}^{(r,\nu)}(x)-\big([N/r]+
\vartheta_N\big)H_{N-r}^{(r,\nu)}(x)\\&=H_{N+r}^{(r,\nu)}(x)-Y_\nu H_N^{(r,\nu)}(x) .
\end{align*}
This yields
\begin{align}
Y_\nu H_N^{(r,\nu)}(x)+2x^r\,H_N^{(r,\nu)}(x)=H_{N+r}^{(r,\nu)}(x).\label{rec4}
\end{align}
We apply the operator $Y_\nu$ to the two members of \eqref{rec4} and we use \eqref{rec2}, to obtain
\begin{align*}
Y^2_\nu H_N^{(r,\nu)}(x)+2  x^rY_\nu\,H_N^{(r,\nu)}(x)=2[N]_\nu H_{N}^{(r,\nu)}(x).
\end{align*}
\end{proof}
We introduce the parabosonic creation and annihilation operators
\begin{align}
a=\frac{1}{\sqrt{2}}\big(Y_\nu+x^r\big),\quad a^\dagger =\frac{1}{\sqrt{2}}\big(-Y_\nu+x^r\big).
\end{align}
These operators have the commutation relations
\begin{align}
&[a,a^\dagger]=1+\frac{1}{r}\sum_{s=0}^{r-1}(2\nu+2s+1-r)\big(\Pi_s(2r)-\Pi_{r+s}(2r)\big)\\&
a\Pi_s(2r)=\Pi_{r+s}(2r)a,\quad \Pi_s(2r)a^\dagger=a^\dagger\Pi_{r+s}(2r).
\end{align}
The Hamiltonian $H_0$ assumes the form
\begin{equation}
H_0=-\frac{1}{2}Y_\nu^2+\frac{1}{2}\,x^r.
\end{equation}
Define the Hermite functions on the radial lines by
\begin{equation}
h_N^{(r,\nu)}(x)=\gamma_N^{-1/2}e^{-\frac{x^{2r}}{2}}H_N^{(r,\nu)}(x),
\end{equation}
where $$\gamma_N=\frac{2^{[N/r]}[N]_\nu!}{\zeta_N}.$$
An easy computation using Propositions 3.2 and 3.3 leads to the following results:
\begin{align}
&a h_N^{(r,\nu)}(x)=\sqrt{[N]_\nu}h_{N-r}^{(r,\nu)}(x),\\&a^\dagger h_N^{(r,\nu)}(x)=\sqrt{[N+r]_\nu}h_{N+r}^{(r,\nu)}(x).
\end{align}
It immediately follows that
\begin{align}
&H h_N^{(r,\nu)}(x)=([N/r]+\frac{1}{2}\nu_s)h_{N}^{(r,\nu)}(x).
\end{align}\\ Proceeding similarly as \cite{Bouz}, we can decompose every function $f:\delta\rightarrow \mathbb{C}$ uniquely in the form
$$f=\sum^{r-1}_{j=0}f_j,\quad f_j=\Pi_j(f).$$
It is clearly that the function $f_j$ satisfies  $f_j(\omega_r^{2}x)=\omega_r^jf_j(x)$. Then it can be identified with a function defined on the real line. Thus facts enables us to extend  the Schwartz space $\mathcal{S}(\mathbb{R})$ to the spaces $\mathcal{S}(\delta)$ of functions defined on the radial lines $\delta$. It is easily  seen that the Dunkl operator $Y_\nu $ map the space $\mathcal{S}(\delta)$ into itself.\\ Let consider the inner product in radials lines
\begin{align}  \label{p1}
\left\langle f,g\right\rangle _{\nu}=
\sum_{j=0}^{r-1}\int_{\mathbb{R}}f(\omega_r^jx)\overline{g(\omega_r^jx)} |x|^{2\nu}dx, \quad \omega_r=e^{\frac{2i\pi}{r}}.
\end{align}
\indent
We denote by $L_{\nu}^2(\delta)$  the space of measurable functions $f$ on $\delta$ satisfying
\begin{equation}
\sum_{j=0}^{r-1}\int_{\mathbb{R}}|f(\omega_r^jx)|^2
|x|^{2\nu}dx<\infty.
\end{equation}
As a direct consequence we notice that the projections $\Pi_j,$ $j=0,\,\dots,\,2r-1,$ are self-adjoint that is
\begin{equation}
\left\langle \Pi_jf,g\right\rangle _{\nu}=\left\langle f,\Pi_jg\right\rangle _{\nu}.
\end{equation}
One can also verify that the multiplication operator by $x^r$ is also self-adjoint
\begin{equation*}
\left\langle x^rf,g\right\rangle _{\nu}=\left\langle f,x^rg\right\rangle _{\nu}.
\end{equation*}
\begin{lemma}
Let $f,g\in \mathcal{S}(\delta)$. Then
\begin{equation*}
\left\langle Y_\nu f,g\right\rangle _{\nu} =-\left\langle f, Y_\nu g\right\rangle _{\nu}.
\end{equation*}
\end{lemma}

\begin{proof}
We have
\begin{eqnarray*}
\left\langle Y_\nu,g\right\rangle _{\nu}
&=&\sum_{j=0}^{r-1}\int_{\mathbb{R}}f'(\omega_r^jx)
\overline{g(\omega_r^jx)%
}|x|^{2\nu-r+1}dx \\
&+&\sum_{s=0}^{r-1}(2\nu+1+s-r)\left\langle \frac{1}{rx^r}\Pi_{r+s}f,g\right\rangle _{\nu}-\sum_{s=0}^{r-1}s
\left\langle \frac{1}{rx^r}\Pi_{s}f,g\right\rangle _{\nu} \\
&=&(I)+(II)-(III)\\
\end{eqnarray*}
Performing integration by parts, the first term $(I)$ becomes
\begin{eqnarray*}
(I)&=&\frac{1}{r}\sum_{j=0}^{r-1}\omega_{r}^{-j} \left\{ \lim_{x\rightarrow
\infty }\left[ f(x)\overline{g(x)}|x|^{2\nu-r+1}\right] -\lim_{x\rightarrow -\infty}\left[
f(x)\overline{g(x)}|x|^{2\nu-r+1}\right] \right\} \\
&&-\frac{1}{r}\sum_{j=0}^{r-1}\int_{\mathbb{R}}f(\omega_{r}^{j}x)\overline{( g'%
+\frac{\alpha}{x}g)(\omega_{r}^{j}x)}|x|^{2\nu-r+1}dx \\&=&-\frac{1}{r}\sum_{j=0}^{r-1}\int_{\mathbb{R}}f(\omega_{r}^{j}x)\overline{( g'%
+\frac{2\nu-r+1}{x}g)(\omega_{r}^{j}x)}|x|^{2\nu-r+1}dx\\&=&-\left\langle f,\frac{d g}{dx^r}+\frac{2\nu-r+1}{rx^r}g\right\rangle _\nu.
\end{eqnarray*}%
For the second and the third terms we use the fact that $$\left\langle \frac{1}{rx^r}\Pi_{s}f,g\right\rangle _{\nu}=\left\langle f,\frac{1}{rx^r}\Pi_{r+s}g\right\rangle _{\nu}.$$ This yields
\begin{eqnarray*}
(II) &=&\sum_{s=0}^{r-1}(2\nu+1+s-r)\left\langle f,\frac{1}{rx^r}\Pi_{s}g\right\rangle _{\nu}\\
(III) &=&\sum_{s=0}^{r-1}s
\left\langle f, \frac{1}{rx^r}\Pi_{r+s}g\right\rangle _{\nu}
.
\end{eqnarray*}%
Combining these equations we get
\begin{equation*}
\left\langle Y_\nu,g\right\rangle _{\nu}=-\left\langle Y_\nu,g\right\rangle _{\nu}.
\end{equation*}
\end{proof}

Thus from \eqref{e9}, the system $\{h_{N}^{(r,\nu)}(x)\}$ is an orthonormal set in $L^2_\nu(\delta)$ and it is complete by same argument used to prove that the classical Hermite functions
form a complete orthogonal set in $L^2(\mathbb{R},dx)$ see \cite{Ah}.
\begin{theorem}The operator $H$ with domain $D(H)=\mathcal{S}(\delta)$
is essentially self-adjoint; the spectra of its closure
is discrete and given by $$\frac{1}{2}\big([N]_\nu+[N+r]_\nu\big), N=0,\,1,\,2,\dots\,\,.$$
\end{theorem}
Now we will give a supersymmetric oscillator. To this end we change $x$ by $x^r$ and the standard derivative $ \frac{d}{dx}$ by the Dunkl operator $Y_\nu$ in the expression of the supercharge defined in \eqref{char1}, to get the new supercharge
\begin{equation}
Q=\frac{1}{\sqrt{2}}\big(Y_\nu R_r+x^r\big)
\end{equation}
where $R_r$ is the reflection given 
\begin{equation}
R_r=\sum_{s=0}^{r-1}(\Pi_s-\Pi_{r+s}).
\end{equation} 
A straightforward computation shows that 
 \begin{align}Y_\nu R_r=-R_rY_\nu, \quad x^rR_r=-R_rx^r, \quad R^2=1.
\end{align}
After evaluating $Q^2$, we get the following form for a supersymmetric Hamiltonian
\begin{equation}
H=Q^2=H_0-\frac{1}{2}[Y_\nu,x^r]R_r.\label{spe}
\end{equation}
The spectrum of $H$ is easily obtained by using \eqref{spe}, we have
\begin{equation}
Hh_{N}^{(r,\nu)}(x)=[N/r]h_{N}^{(r,\nu)}(x).
\end{equation}
The spectrum of $H$  consist only for the number $[N/r]$ starting with zero. Each level is degenerate.

\end{document}